\documentclass[12pt]{amsart}
\usepackage{amssymb}
\usepackage{float}
\usepackage{color}
\usepackage{ graphicx}

\usepackage[all]{xy}
\usepackage[colorlinks]{hyperref}
\newcommand{\R}{{\mathbb R}}
\newcommand{\C}{{\mathbb C}}
\renewcommand{\P}{{\mathbb P}}
\newcommand{\bZ}{\mathbb{Z}}
\newcommand{\bR}{\mathbb{R}}
\newcommand{\bC}{\mathbb{C}}

\newcommand{\M}{{\mathcal M}}
\newcommand{\cM}{\mathcal{M}}
\newcommand{\cN}{\mathcal{N}}

\renewcommand{\k}{{\mathfrak k}}

\newcommand{\su}{{\mathfrak{su}}}
\newcommand{\fs}{\mathfrak{s}}
\newcommand{\ft}{\mathfrak{t}}
\newcommand{\fI}{\mathfrak{I}}

\newcommand{\Si}{{\Sigma}}

\renewcommand{\o}{\operatorname}
\newcommand{\SU}{\o{SU}}
\newcommand{\U}{\o{U}}
\renewcommand{\H}{\o{H}}

\newcommand{\tr}{\o{tr}}
\renewcommand{\i}{ {\scriptscriptstyle\sqrt{-1}}\, }
\newcommand{\Hom}{\o{Hom}}
\newcommand{\la}{\langle}
\newcommand{\ra}{\rangle}
\newcommand{\lto}{{\longrightarrow}}

\newtheorem{thm}{Theorem}[section]

\newtheorem{lm}[thm]{Lemma}

\newtheorem{pro}[thm]{Proposition}
\newtheorem{prop}[thm]{Proposition}
\newtheorem{rem}[thm]{Remark}

\newtheorem{defin}[thm]{Definition}

\title{The $SU(2)$-character variety of the closed surface of genus 2}

\author{Nan-Kuo Ho}
\address{Department of Mathematics, National Tsing Hua University, Hsinchu 300, and National Center for Theoretical Sciences, Taipei 106, Taiwan}
\email{nankuo@math.nthu.edu.tw}

\author{Lisa C. Jeffrey}
\address{Department of Mathematics, University of Toronto, Toronto, Canada}
\email{jeffrey@math.toronto.edu}

\author{Khoa Dang Nguyen}
\address{Department of Mathematics, University of Michigan, Ann Arbor, MI, USA}

\author{Eugene Z. Xia}
\address{Department of Mathematics, National Cheng Kung University, Tainan 70101, Taiwan}
\email{ezxia@ncku.edu.tw}

\thanks{}

\begin{document}

\begin{abstract}
We study the symplectic geometry of the $\SU(2)$-representation variety of the compact oriented surface of genus 2. We use the Goldman flows to identify subsets of the moduli space with corresponding subsets of $\P^3(\bC)$. We also define and study two antisymplectic involutions on the moduli space and their fixed point sets.
\end{abstract}

\maketitle
\section{Introduction}

Let $\Sigma$ be a compact orientable surface and $K$ a compact Lie group.  From these two ingredients comes the character variety: the moduli space of conjugacy classes of representations of the fundamental group of $\Sigma$ into $K$.  If $\Sigma$ is provided with a smooth structure, then we obtain the moduli space of gauge equivalence classes of flat $K$-connections on $\Sigma$.  If in addition $\Sigma$ is provided with a complex structure $J$, i.e. $(\Sigma, J)$ is a Riemann surface, then we obtain the moduli space of isomorphism classes of semi-stable holomorphic vector bundles on $\Sigma$.  These three objects play central roles in symplectic, differential and K\"ahler geometry, respectively.
With suitable restrictions, these are homeomorphic as topological spaces.

%
%
This paper focuses on the character variety when $\Sigma$ is the compact oriented surface of genus $2$ and $K=\SU(2)$. Let $\pi_1(\Sigma)$ be the fundamental group of $\Sigma$ and $\Hom(\pi_1(\Sigma), K)$ be the space of homomorphisms from $\pi_1(\Sigma)$ to $K$. The representation variety $\Hom(\pi_1(\Sigma), K)$ inherits a topology from $K$ and $K$ acts on $\Hom(\pi_1(\Sigma), K)$ by equivalence (conjugation) of representations.  The character variety is the quotient
$$
\M = \Hom(\pi_1(\Sigma),K)/K.
$$
If $\rho \in \Hom(\pi_1(\Sigma), K)$, we denote its image in $\M$ as $[\rho]$.
$\M$ contains a Zariski open set of irreducible $\pi_1(\Sigma)$-representation classes and we denote this open set $\M^i$.  

$\M$ is homeomorphic  to the moduli space of semi-stable vector bundles on $(\Sigma, J)$ \cite{NS65}.  This interpretation provides a complex structure to $\M$ which depends on the choice of $J$.  Then ([Theorem 2, \cite{NR69}])
\begin{thm}[Narasimhan-Ramanan]\label{thm:NR}
The moduli space of S-equivalence classes of semi-stable vector bundles of rank 2 with trivial determinant on $\Sigma$ is isomorphic to $\P^3(\C)$.
\end{thm}

The proof of Theorem~\ref{thm:NR} is algebro-geometric.  Later Choi
 provided an alternative proof of the fact that $\M$ is homeomorphic to $\P^3(\C)$ via the moduli space of (singular) flat elliptic structures on $\Sigma$ \cite{Ch11}.

As a variety, $\M$ may be singular,
but $\M$ contains an open dense subset $\M^i$ that has a natural symplectic structure $\omega$ \cite{Go84}.  This character variety perspective gives the most explicit and concrete description of the symplectic structure.  The space $\M$ with its open dense symplectic $\M^i$ offers an interesting example in low dimensional topology and symplectic geometry. In this paper, we study the symplectic geometry of $\M^i$ in the most explicit manner and its implications on $\M$.

In our particular case of $K=\SU(2)$ with genus $g=2$, the character variety is indeed singular, but it is a topological manifold homeomorphic to $\P^3(\C)$. We had hoped to find a proof of this result using symplectic geometry and toric geometry, since the standard moment polytope of $\P^3(\C)$ is the 3-simplex and may be identified with the image of the moduli space under the Goldman flows with suitable modifications
\cite{Go86,JW92,JW94}. This article describes how far we were able to proceed with this
program.

Here is an outline. Section 2 describes the Goldman flows, while Section 3 describes the symplectic structure on the moduli space. Section 4 uses the Goldman flows to identify subsets of the moduli space with subsets of the projective  space.  Anti-symplectic involutions are as important to symplectic geometry as complex conjugation (anti-holomorphic involutions) is to complex geometry.  Section 5 describes two antisymplectic involutions on the moduli space, the first being compatible with the Goldman flows in the Duistermaat sense \cite{Du83} and the other is not compatible but suggests another compatible flow.

\thanks{The first author is partially supported by Ministry of Science and Technology of Taiwan, grant 105-2115-M-007 -006. The second author is partially supported by a grant from NSERC. The third author is partially supported by a UTEA grant of University of Toronto. The fourth author is partially supported by Ministry of Science and Technology of Taiwan, grant 105-2115-M-006-006.}
\section{$\M$ in coordinates}
We begin by describing $\M$ explicitly.  Let $\Sigma$ be the genus 2 closed surface.  Define the commutator operator $[A,B] = ABA^{-1}B^{-1}$.  Then the fundamental group of $\Sigma$ has a presentation:
$$
\pi_1(\Sigma) = \la A_1, B_1, A_2, B_2 \ | \ \prod_{i=1}^2 [A_i, B_i] \ra.
$$
For any representation $\rho\in\Hom(\pi_1(\Sigma),K)$, denote $g_i=\rho(A_i)$, $h_i=\rho(B_i)$, $i=1,2$. The representation space can be realized as
$$
\Hom(\pi_1(\Sigma),K)  = \{(g_1, h_1, g_2, h_2) \in K^4 : \prod_{i=1}^2 [g_i, h_i] = I\}
$$
where $I$ is the identity element of $K$ and the representation variety is $\M = \Hom(\pi_1(\Sigma),K)/K$.
Hence a point in $\M$ is represented by $[(g_1,h_1,g_2,h_2)]$ or simply $[g_1;h_1; g_2; h_2]$.
\begin{defin}
For any representation $\rho\in\Hom(\Gamma,K)$, we say $\rho$ is abelian if its image $\rho(\Gamma)$ in $K$ is abelian.
\end{defin}
For example, $\rho \in \Hom(\pi_1(\Sigma),K)$ is abelian if $\rho=(g_1,h_1,g_2,h_2)$ is abelian, i.e. $g_1,h_1,g_2,h_2$ all commute with each other.

\subsection{The trace and angle functions}
$ $

A function $f : \Hom(\pi_1(\Sigma),K) \lto \C$ descends to a function (which we also name) $f: \M \lto \C$ if and only if $f$ is $K$-conjugation invariant.  Since the trace functions on $K$ are conjugation invariant, $\M$ has trace coordinates (see e.g. \cite{Go97}).  Let $A \in \pi_1(\Sigma)$ and $\rho \in \Hom(\pi_1(\Sigma),K)$.  Then we have the trace function
$$
\tr_A : \M \to \R, \ \ \ \tr_A([\rho]) = \tr(\rho(A)).
$$
For our purposes, we use the modified trace coordinates \cite{JW92}
\begin{equation}\label{eq:trace}
f_A : \M \to \R, \ \ \ f_A([\rho]) = \frac{\cos^{-1}(\tr(\rho(A)/2))}{\pi}.
\end{equation}

\subsection{Free group on two generators}
$ $

Let $F_2 = \la A, B\ra$ be the free group on two generators, which can be understood as the fundamental group of a pair of pants or a three-holed sphere. Consider the representation variety of $F_2$:
$$
 \Hom(F_2, K)/K = K^2/K.
$$
Using the trace functions, we obtain a coordinate system for $\Hom(F_2,K)/K$ \cite{Go09}
$$
\Psi : \Hom(F_2,K)/K \lto \R^{\times 3}, \ \ \ \Psi([\rho]) = (f_A[\rho], f_B[\rho], f_{AB} ([\rho]).
$$
\begin{prop}
$\Psi$ identifies $\Hom(F_2, K)/K $ with the tetrahedron $\tilde\Delta\subset \R^{\times 3}$ having the vertex set
$$
V = \{(0,0,0), (0,1,1), (1,0,1), (1,1,0)\}.
$$
Moreover, $\Psi([\rho]) \in \partial \tilde\Delta$ if and only if $\rho$ is abelian (i.e. if $\rho(A)$ and $\rho(B)$ commute).
\end{prop}
\begin{proof}
See [Prop. 3.1, \cite{JW92}] and [Section 4, \cite{Go09}].
\end{proof}
We denote by $S$ and $L$ the sets of the interiors of the faces and edges of $\tilde\Delta$, respectively. When there is no confusion, we will use the notations $\tilde\Delta$ and $K^2/K$ interchangeably.

\section{The symplectic structure on $\M^i$}\label{sec:symp}

The group $K$ is compact and acts linearly on $\C^2$ by definition. Let $\k = \su(2)$ be its Lie algebra and $B$ be the Killing form on $\k$.  Denote  $T=\U(1)$.  Then $T$ is also the (diagonal) maximal torus of $K$.
Then there is an adjoint $K$-action on $\k$.

\begin{defin}
$$\M^i = \{\rho \in \Hom(\pi_1(\Sigma), K) : \rho \text{ is not abelian }\}/K.$$
\end{defin}
Recall that a representation is called irreducible if its stabilizer has minimal dimension. Since $\SU(2)$ admits property CI (ref: \cite{S}), $\rho\in\cM^i$ iff $\rho$ is an irreducible representation, hence the notation $\cM^i$.

\begin{prop}
$\M \setminus \M^i = T^4/W$,
where $W = N(T)$ is the Weyl group of $K$.
\end{prop}

\begin{proof}
The space $\M \setminus \M^i$ corresponds to abelian representations. In other words, if $[\rho] = [g_1;h_1;g_2;h_2]  \in \M \setminus \M^i$, then $g_1, h_1, g_2, h_2$ all commute with each other,
so they belong to the same  maximal torus. \end{proof}

\begin{prop}\label{prop:Mi_dense}
$\M^i$ is dense in $\M$.
\end{prop}
\begin{proof}
Let
$$
Q = \{[g_1;h_1;g_2;h_2] \in T^4/W : g_1, g_2, h_1, h_2 \not\in \{\pm I\}\}.
$$
Then $Q$ is dense in $T^4/W$ because $(T^4/W) \setminus Q$ is a union of tori of lower dimensions modulo $W$.
Let  $[\rho] = [g_1;h_1; g_2; h_2] \in Q$.  We may assume that $g_j$ and $h_j$ are diagonal for $j = 1,2$, defining a unique maximal torus $T \subset K$.  Let
$k : [0,1] \to K$ be a continuous path such that $k(0) = I$ and $k(t) \not\in T$ for all $t \neq 0$.
Let
$$\rho_t =  (k(t) g_1k(t)^{-1}, k(t) h_1k(t)^{-1},  g_2, h_2).$$
Then it is immediate that $[\rho_t] \in \M^i$ for $t \neq 0$ and $[\rho_0] = [\rho]$.
Hence $\M^i$ is dense in $\M^i \cup Q$.  Since $Q$ is dense in $T^4/W$ and $\M = \M^i \cup (T^4/W)$, $\M^i$ is dense in $\M$.
\end{proof}

The adjoint $K$-action induces a $\pi_1(\Sigma)$-action on $\k$, making $\k$ a $\pi_1(\Sigma)$-module.
The tangent space to $\M$ at an irreducible representation $[\rho]$ is then the $\pi_1(\Sigma)$-module cohomology $\H^1(\pi_1(\Sigma),\k)$ and :
\begin{thm}[Goldman \cite{Go84}]
Suppose $[\rho] \in \M^i$.  Then there is a perfect anti-commuting pairing
$$
\omega : \H^1(\pi_1(\Sigma),\k) \times \H^1(\pi_1(\Sigma),\k) \lto \H^2(\pi_1(\Sigma),\R)
$$
via Poincar\'e duality on cocycles and $B$ on the coefficient vector space $\k$.
\end{thm}
Since $\Sigma$ is compact, $\H^2(\pi_1(\Sigma),\R) \cong \R$ and $\omega$ is our desired symplectic structure on $\M^i$.

\section{The $T^3$-action}

Let $T^3 = \U(1)^{\times 3}$.
The closed genus two surface may be decomposed into two three-holed spheres, glued along three boundary circles $C_1, C_2, C_3$.
Throughout this paper, we consider the decomposition such that no $C_i$ separates $\Sigma$ into disjoint components, see Figure 1.

\begin{figure}[h]\label{pantsdecomposition}
\scalebox{0.2}{\includegraphics{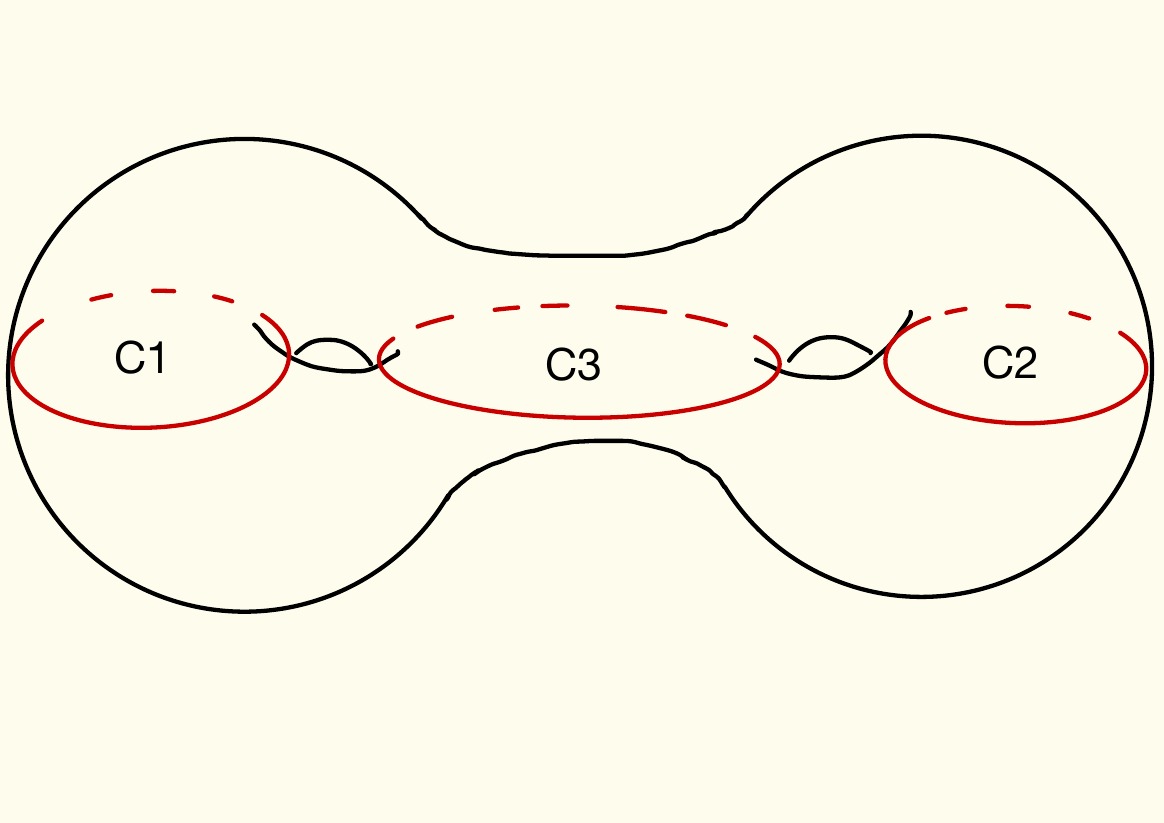}}
\caption{The circles $C_1,C_2,C_3$}
\end{figure}

Denote by $f_i$ the modified trace function associated with $C_i$, i.e. $f_i=f_{C_i}$ and let
$$
\mu : \M \lto \R^{\times 3}, \ \ \ \mu([\rho]) = (f_1([\rho]), f_2([\rho]), f_3([\rho])).
$$
Then $C_i$ corresponds to $[g_1;h_1;g_2;h_2]\in\M$ via $C_1, C_2, C_3$ representing $h_1,h_2,h_1h_2$ respectively.
\begin{prop}
$\mu(\M) = \tilde\Delta$.
\end{prop}
\begin{proof}
Observe that the subset $U := \{[I;h_1;I;h_2] : h_1, h_2 \in K\} \subset \M$ and $\mu(U) = \tilde\Delta$, so $\mu(\M)\supset \tilde\Delta$. In fact, $\mu(\M)=\tilde\Delta$, see \cite{JW92} for more detail.
\end{proof}
Denote by $\tilde\Delta^\circ$ the interior of $\tilde\Delta$ and $\M^\circ = \mu^{-1}(\tilde\Delta^\circ)$. Then $\mu : \M^\circ \to \R^{\times 3}$ is the moment map of a Hamiltonian $T^3$-action which will be described next.

\subsection{Hamiltonian $T^3$-action}\label{subsec:T3}
$ $

Goldman's flows \cite{Go86} define an $\bR^3$-action on $\M$. Jeffrey-Weitsman \cite{JW92} modified Goldman's moment map to $\mu$, which gives a Hamiltonian $T^3$-action on $\M^\circ$ as follows.
The torus action corresponding to $C_1$ and $C_2$ is
$$T^2:\M^\circ \to\M^\circ ,\quad [g_1;h_1;g_2;h_2]\mapsto [g_1e^{t_1\xi_1};h_1;g_2e^{t_2\xi_2};h_2], $$
where $ t_1,\ t_2\ \in \bR$, $ h_1=e^{\xi_1},\ h_2=e^{\xi_2}$.
The torus action corresponding to $C_3$ is less transparent.

Let $X=h_2h_1-(h_2h_1)^{-1}$ and $Y=h_1h_2-(h_1h_2)^{-1}$, then $X, Y\in \k=\su(2)$. The torus corresponding to $C_3$ acts on $\M^\circ$ as follows:
$$T: K^4 \to K^4,\quad (g_1,h_1,g_2,h_2)=(e^{tX}g_1,h_1,e^{tY}g_2,h_2), \mbox{where}\ t\in \bR,$$
which descends to an action on $\M^\circ$:
\begin{pro}[Goldman \cite{Go86}]
If $[g_1,h_1][g_2,h_2]=I$, then $[e^{tX}g_1,h_1][e^{tY}g_2,h_2]=I$.
\end{pro}
\begin{proof}
We only need to show that $e^{tX}h_2=h_2e^{tY}$ and $h_1e^{tX}=e^{tY}h_1$. Indeed, if these two equalities are true, then
\begin{eqnarray*}
[e^{tX}g_1,h_1][e^{tY}g_2,h_2]&=&e^{tX}g_1h_1g_1^{-1}e^{-tX}h_1^{-1}e^{tY}g_2h_2g_2^{-1}e^{-tY}h_2^{-1}\\
&=&e^{tX}g_1h_1g_1^{-1}h_1^{-1}g_2h_2g_2^{-1}e^{-tY}h_2^{-1}\\
&=&e^{tX}g_1h_1g_1^{-1}h_1^{-1}g_2h_2g_2^{-1}h_2^{-1}e^{-tX}=I.
\end{eqnarray*}

To show that $h_1e^{tX}=e^{tY}h_1$, recall that $e^{tX}=I+tX+\frac{t^2}{2}X^2+\cdots$ and $e^{tY}=I+tY+\frac{t^2}{2}Y^2+\cdots$, so
\begin{eqnarray}
h_1e^{tX}&=&h_1+th_1(h_2h_1-(h_2h_1)^{-1})+\frac{t^2}{2}h_1(h_2h_1-(h_2h_1)^{-1})^2+\cdots,\label{eq:htx}\\
e^{tY}h_1&=&h_1+t(h_1h_2-(h_1h_2)^{-1})h_1+\frac{t^2}{2}(h_1h_2-(h_1h_2)^{-1})^2h_1+\cdots.\label{eq:tyh}
\end{eqnarray}
Since $(a-a^{-1})^n=a^n-(^n_1)a^{n-2}+(^n_2)a^{n-4}+\cdots\pm a^{-n}$, to check that (\ref{eq:htx})$=$(\ref{eq:tyh}), we only need to show that $h_1(h_2h_1)^n=(h_1h_2)^nh_1$ for all $n\in \bZ$. However, this is indeed true, because
$$h_1(h_2h_1\cdots h_2h_1)=(h_1h_2\cdots h_1h_2)h_1,\quad h_1(h_1^{-1}h_2^{-1}\cdots h_1^{-1}h_2^{-1})=(h_2^{-1}h_1^{-1}\cdots h_2^{-1}h_1^{-1})h_1.$$

Similarly, to show that $e^{tX}h_2=h_2e^{tY}$, we only need to show that $(h_2h_1)^nh_2=h_2(h_1h_2)^n$ for all $n\in \bZ$ and clearly this is also true.
\end{proof}

This $T^3$-action on $\M^\circ$ is Hamiltonian with respect to the symplectic form $\tfrac{-1}{2\pi} \omega$ and has $\mu:\M^\circ\to \bR^{\times 3},\ \mu([\rho]) = (f_1([\rho]), f_2([\rho]), f_3([\rho]))$ as its moment map \cite{JW92}.

\subsection{The relations between $\M^\circ$, $\M^i$, and $\M$}
$ $

\begin{lm}\label{lm:boundary}
Let $\partial \tilde\Delta$ denote the boundary of $\tilde\Delta$, which is the disjoint union $S\cup L\cup V$. Then
$$\M\setminus\M^\circ=\mu^{-1}(\partial \tilde\Delta)=\{[g_1;h_1;g_2;h_2]\in \M : [h_1,h_2]=I\}$$
\end{lm}
\begin{proof}
This follows from \cite{Go84} or \cite{JW92}.
\end{proof}

\begin{prop}
$$
\M^\circ \subsetneqq \M^i.
$$
\end{prop}
\begin{proof}
Suppose
$[\rho] = [g_1;h_1;g_2;h_2] \in \mu^{-1}(\tilde\Delta^\circ)$.
It follows from Lemma \ref{lm:boundary} that $h_1$ does not commute with $h_2$.  Hence $\rho$ is not abelian, i.e. $[\rho] \in \M^i$. On the other hand, consider $[g_1;I;g_2;I]\in\M\setminus \M^\circ$. For any $g_1,\ g_2\in K\setminus \{\pm I\}$ such that $g_1$ does not commute with $g_2$ (the commutator relation is satisfied automatically), $(g_1,I,g_2,I)$ is certainly not abelian, i.e. $[g_1;I;g_2;I]$ belongs to $\M^i$. We conclude that $\M^\circ \subsetneqq \M^i.$
\end{proof}

\begin{prop}\label{prop:face}
If $[\rho]=[g_1;h_1;g_2;h_2]\in \mu^{-1}(\partial \tilde\Delta)$ is an irreducible representation, then at least one of $h_1, h_2, h_1h_2$ is $\pm I$, i.e. it belongs to $\mu^{-1}(L)\cup \mu^{-1}(V)$, thus $\mu^{-1}(S)$ consists of abelian representations only.  Hence $\dim(\mu^{-1}(S)) \le 4$
\end{prop}

\begin{proof}
Since $[g_1;h_1;g_2;h_2]\in\mu^{-1}(\partial\tilde\Delta)$, $[h_1,h_2]=I$, i.e. they are in the same maximal torus. Without loss of generality, we can assume $h_1=\left(
\begin{smallmatrix}
e^{i\theta_1} & 0
\\
0 & e^{-i\theta_1}
\end{smallmatrix}
\right)$, $h_2=\left(
\begin{smallmatrix}
e^{i\theta_2} & 0
\\
0 & e^{-i\theta_2}
\end{smallmatrix}
\right)$. 

Denote $g_1=\left(
\begin{smallmatrix}
w_1 & z_1
\\
-\bar{z}_1 & \bar{w}_1
\end{smallmatrix}
\right)$, $g_2=\left(
\begin{smallmatrix}
w_2 & z_2
\\
-\bar{z}_2 & \bar{w}_2
\end{smallmatrix}
\right)$ with $w_j\bar{w}_j+z_j\bar{z}_j=1$ for $j=1,2$.
Then
$$ [g_1,h_1]=\left(
\begin{smallmatrix}
w_1\bar w_1+z_1\bar z_1e^{-2i\theta_1} & w_1z_1(1-e^{2i\theta_1})
\\
\bar w_1\bar z_1(e^{-2i\theta_1}-1) & w_1\bar w_1+z_1\bar z_1e^{2i\theta_1}
\end{smallmatrix}
\right), \quad [g_2,h_2]=\left(
\begin{smallmatrix}
w_2\bar w_2+z_2\bar z_2e^{-2i\theta_2} & w_2z_2(1-e^{2i\theta_2})
\\
\bar w_2\bar z_2(e^{-2i\theta_2}-1) & w_2\bar w_2+z_2\bar z_2e^{2i\theta_2}
\end{smallmatrix}
\right).$$
So $[g_1,h_1]=[g_2,h_2]^{-1}$ is equivalent to
\[
w_1\bar w_1+z_1\bar z_1e^{-2i\theta_1}=w_2\bar w_2+z_2\bar z_2e^{2i\theta_2}\quad \mbox{and}\quad
w_1z_1(1-e^{2i\theta_1})=-w_2z_2(1-e^{2i\theta_2}).
\]
Using the condition that $w_j\bar w_j+z_j\bar z_j=1$ and separating the real and imaginary parts of the first equality, the above two equalities are equivalent to the follow three equalities:
\begin{eqnarray}
w_1\bar w_1+(1-w_1\bar w_1) \cos(-2\theta_1)&=&w_2\bar w_2+(1-w_2\bar w_2) \cos(2\theta_2), \nonumber\\
(1-w_1\bar w_1) \sin (-2\theta_1)&=&(1-w_2\bar w_2) \sin (2\theta_2), \nonumber\\
w_1z_1(1-e^{2i\theta_1})&=&-w_2z_2(1-e^{2i\theta_2}).\label{eq:third}
\end{eqnarray}
To simplify notation, let $w_j\bar w_j=c_j$, then $z_j\bar z_j=1-c_j$, $0\leq c_j\leq 1$, $j=0,1$.
Then the first equality becomes
\begin{equation}\label{eq:one}(1-c_1)(1-\cos 2\theta_1)=(1-c_2)(1-\cos 2\theta_2), \end{equation}
and the second equality becomes
\begin{equation}\label{eq:two} -(1-c_1)\sin 2\theta_1=(1-c_2)\sin 2\theta_2.\end{equation}
Equation (\ref{eq:third}) has complex numbers on both sides, so their lengths must be equal, which gives
\begin{eqnarray}
 c_1(1-c_1)(1-\cos 2\theta_1)=c_2(1-c_2)(1-\cos 2\theta_2). \label{eq:fourth}
 \end{eqnarray}

If $[g_1;h_1;g_2;h_2]$ is an irreducible representation, where $h_1, h_2$ are diagonal matrices, then at least one of $c_1, c_2$ is not $1$, because it means the off-diagonal entry of $g_1$ or $g_2$ is nonzero. Thus, we divide the discussion into the following $4$ cases, $(c_1\neq 1, c_2\neq 1)$, $(c_1\neq 1, c_2=1)$, $(c_1=1, c_2\neq 1)$, and $(c_1=1, c_2=1)$.

\noindent {\em Case 1:} $(c_1\neq 1, c_2\neq 1)$.
\begin{enumerate}
\item $\cos 2\theta_1\neq 1$.
Then equation (\ref{eq:one}) implies that $\cos 2\theta_2\neq 1$.
\begin{enumerate}
\item $c_1=0$.\\
Then equation (\ref{eq:fourth}) implies that $c_2=0$ because $1-c_2\neq 0$, and equation (\ref{eq:two}) implies that $-\sin 2\theta_1=\sin 2\theta_2$, and equation (\ref{eq:one}) implies that $\cos 2\theta_1=\cos 2\theta_2$. This two conditions implies that $2\theta_1=-2\theta_2$ or $2\theta_1=2\pi-2\theta_2$. Thus, $\theta_1+\theta_2=0$ or $\pi$. Thus, $h_1h_2=\left(
\begin{smallmatrix}
e^{i(\theta_1+\theta_2)} & 0
\\
0 & e^{-i(\theta_1+\theta_2)}
\end{smallmatrix}
\right)=
I$ or $-I$.

\item $c_1\neq 0$.\\
Then equation (\ref{eq:fourth}) also implies that $c_2\neq 0$, because the left hand side of equation (\ref{eq:fourth}) is nonzero, $1-c_2\neq 0$, $\cos 2\theta_2\neq 1$. Substituting equation (\ref{eq:one}) into equation (\ref{eq:fourth}), we have $c_1(1-c_2)(1-\cos 2\theta_2)=c_2(1-c_2)(1-\cos 2\theta_2)$, and since there is no zero term, we conclude $c_1=c_2$. Thus, equation (\ref{eq:one}) implies that $\cos 2\theta_1=\cos 2\theta_2$ and equation (\ref{eq:two}) implies $-\sin2\theta_1=\sin 2\theta_2$. This is the same as above, which gives $h_1h_2=I$ or $-I$.
\end{enumerate}
\item $\cos 2\theta_1=1$.\\
Then $\cos \theta_1=\pm 1$ and so $h_1=I$ or $-I$.
\end{enumerate}

\noindent {\em Case 2:} $(c_1\neq 1, c_2=1)$.
\begin{enumerate}
\item $c_1=0$.\\
Then equation (\ref{eq:one}) implies that $\cos 2\theta_1$=1, i.e. $\cos \theta_1=\pm 1$, and so $h_1=I$ or $-I$.
\item $c_1\neq 0$.\\
Then equation (\ref{eq:third}) implies that $w_1z_1(1-e^{2i\theta_1})=0$ because $1-c_2=0$ meaning $z_2=0$. This implies that $e^{2i\theta_1}=1$ because $c_1\neq 0,1$. Thus $\cos 2\theta_1=1$ and so $h_1=I$ or $-I$.
\end{enumerate}
\noindent {\em Case 3:} $(c_1= 1, c_2\neq1)$.
\begin{enumerate}
\item $c_2=0$.\\
Then equation (\ref{eq:one}) implies that $\cos 2\theta_2$=1, i.e. $\cos \theta_2=\pm 1$, and so $h_2=I$ or $-I$.
\item $c_2\neq 0$.\\
Then equation (\ref{eq:third}) implies that $w_2z_2(1-e^{2i\theta_2})=0$ because $1-c_1=0$ meaning $z_1=0$. This implies that $e^{2i\theta_2}=1$ because $c_2\neq 0,1$. Thus $\cos 2\theta_2=1$ and so $h_2=I$ or $-I$.
\end{enumerate}
\noindent {\em Case 4:} $(c_1=1,c_2=1)$.\\
This gives only reducible representations. Since $1-c_1=0=1-c_2$, so $g_1$ and $g_2$ are both diagonal matrices, $g_1$,$h_1$, $g_2$, $h_2$ all commute with each other, i.e. $[g_1;h_1;g_2;h_2]$ is abelian.
\end{proof}

 \begin{prop}\label{prop:edge}
$\dim(\mu^{-1}(L)) \le 4$.
\end{prop}
\begin{proof}
Suppose that $\underline{a}=(a_1,a_2,a_3) \in L$.  Without loss of generality, we may assume $a_1 = 0$.  Then
$h_1=I$ and $$
\mu^{-1}(\underline{a}) = \{[g_1; I; g_2; h_2] \in\M: [g_2,h_2] = I \ \text{ and } \ \tr(h_2) = 2\cos(\pi a_2)\}.
$$
Fix $g_1$.

\noindent {\em Case 1:}  $g_2 \not\in \{\pm I\}$.
Then $\mu^{-1}(\underline{a})$ contains two points
$$[g_1; I; g_2; h_2],  [g_1^{-1}; I; g_2^{-1}; h_2^{-1}].$$   Since $\dim(K) = 3$, $\dim(\mu^{-1}(\underline{a})) \le 3.$

\noindent {\em Case 2:} $g_2 \in \{\pm I\}$.  Without loss of generality, we may assume $g_2 = I$.  Then $[g_1; I; I; h_2] \in \mu^{-1}(\underline{a}
)$ for the entire conjugacy class of $h_2$.  This means
$$
\mu^{-1}(\underline{a}) = \{[g_1; I; I; h_2]\in\M :  g_1, h_2 \in K \text{ and }  \ \tr(h_2) = 2\cos(\pi a_2)\}
$$
By fixing $h_2$, we see that $\dim(\mu^{-1}(\underline{a})) \le 3$.

In both cases, $\dim(\mu^{-1}(\underline{a})) \le 3$.  Since $\dim(L) = 1$, the Proposition follows.
\end{proof}

 \begin{prop}\label{prop:vertex}
If $\underline{a} = (a_1, a_2, a_3) \in V$, then $\mu^{-1}(a)$ is homeomorphic to $\tilde\Delta$.  Hence $\dim(\mu^{-1}(V)) \le 3$.
\end{prop}
\begin{proof}
Without loss of generality, we may assume that $a_1 = a_2 = 0$.  Then
$$
\mu^{-1}(\underline{a}) = \{[g_1; I; g_2; I]\in\M : g_1, g_2 \in K\} = (K \times K) /K \cong \tilde\Delta.
$$

\end{proof}

 \begin{thm}
$\M^\circ$ is open and dense in $\M$.
\end{thm}
\begin{proof}
Since $\M^i$ is smooth, Propositions \ref{prop:face}, \ref{prop:edge}, \ref{prop:vertex} imply that $\M^\circ$ is dense in $\M^i$.  By Proposition \ref{prop:Mi_dense}, $\M^\circ$ is dense in $\M$.
\end{proof}

\subsection{Comparison with $\P^3(\C)$}
$ $

Let $\Lambda=\{t\in T^3|~ t\cdot [\rho]=[\rho], \forall [\rho]\in\M^\circ\}$ be the isotropy group of the Hamiltonian $T^3$-action on $\M^\circ$ described in Section \ref{subsec:T3}. Direct calculation shows that $\Lambda=\{(1,1,1), (-1,-1,-1)\}$. Let $T^3_\Lambda=T^3/\Lambda$. Then the Hamiltonian $T^3$ action induces a Hamiltonian $T^3_\Lambda$-action. Denote the resulting moment map by $\mu_\Lambda$.

Denote by $\Delta$ the standard 3-simplex in $\bR^3$ with vertices at
$$
(0,0,0), (0,0,1), (0,1,0), (1,0,0)
$$
and $\Delta^\circ$ the interior of $\Delta$.

\begin{prop}\label{prop:moment image M}
The moment map image $\mu_\Lambda(\M^\circ)$ is $\Delta^\circ$.
\end{prop}
\begin{proof}
The explicit quotient homomorphism is
$$
P_\Lambda : T^3 \lto T^3_\Lambda, \ \ \ P_\Lambda(t_1,t_2,t_3) = (t_1 t_2, t_2 t_3, t_1 t_3).
$$
Let $\ft$ and $\ft_\Lambda$ be the Lie algebra of $T^3$ and $T^3_\Lambda$, respectively.  Then
$P_\Lambda$ induces homomorphisms
$$P_\Lambda : \ft \lto \ft_\Lambda, \ \ \ P_\Lambda^* : \ft_\Lambda^* \lto \ft^*.$$
The moment map associated with the induced Hamiltonian $T^3_\Lambda$-action satisfies
$$
\mu = P_\Lambda^* \circ \mu_\Lambda.
$$
A direct calculation shows
$$
P_\Lambda =
\left(
\begin{array}{ccc}
 1 & 1 & 0 \\
 0 & 1 & 1 \\
 1 & 0 & 1
\end{array}
\right).
$$
and the moment image of $\mu_\Lambda$ is $(P_\Lambda^*)^{-1}(\tilde\Delta^\circ)=\Delta^\circ$.
\end{proof}

A priori, $T^3_\Lambda$ is only an effective action on $\M^\circ$. However, it is in fact free because $\Lambda$ is the kernel of the action at a generic point.
Note that $T^3_\Lambda$ is isomorphic to $T^3$ but the original $T^3$ action is not effective and is not free.

\begin{thm}\label{thm:trivial}
$$\M^\circ \cong T^3_\Lambda \times \Delta^\circ.$$
\end{thm}
\begin{proof}
$\M^\circ$ is a 6-dimensional symplectic manifold with an effective $T^3_\Lambda$-Hamiltonian action.  Hence $(\M^\circ, \mu_\Lambda, T^3_\Lambda)$ forms a completely integrable system.
Since the action is free, $\M^\circ$ is a $T^3_\Lambda$-bundle over $\Delta^\circ$.
The Theorem then follows from the fact that $\Delta^\circ$ is contractible.
\end{proof}
This Theorem provides a global coordinate system for $\M^\circ$.

On the other hand, recall the standard Hamiltonian $T^3$-action on $\P^3(\C)$.
For $\bold{z} \in \C^4 \setminus \{0\}$, denote by $||\bold{z}||$ its Euclidean norm and $[\bold{z}]$ its image in $\P^3(\C)$.
$\P^3(\C)$ is a completely integrable system with the action
$$
J : T^3 \times \P^3(\C) \lto \P^3(\C), \ \ \ J((t_1,t_2,t_3),[\bold{z}]) = [z_0; t_1 z_1; t_2 z_2; t_3 z_3].
$$
The associated moment map is
$$
\nu : \P^3(\C) \lto \R^{\times 3}, \ \ \ \nu[\bold{z}] = \frac{(|z_1|^2, |z_2|^2, |z_3|^2)}{2||\bold{z}||^2},
$$
and the moment image is the standard 3-simplex $\Delta$. Moreover, we know that $\nu^{-1}(\Delta^\circ)$ is open dense in $\P^3(\C)$ and that the action is free on $\nu^{-1}(\Delta^\circ)$.

\begin{rem}
We now have two completely integrable Hamiltonian systems with the same moment map image. In particular, we have $\M^\circ\cong T^3_\Lambda\times\Delta^\circ\cong T^3\times \Delta^\circ\cong \nu^{-1}(\Delta^\circ)$. Thus, we conclude a (symplectic) identification between an open dense subset $\M^\circ$ of $\M$ and an open dense subset $\nu^{-1}(\Delta^\circ)$ of $\P^3(\C)$.
Moreover, if we define $\mu_\Lambda$ by $\mu=P^*_\Lambda\circ\mu_\Lambda$ on $\M\setminus \M^\circ$ as well, we get $\mu_\Lambda(\M)=\Delta=\nu(\P^3(\C))$.
\end{rem}

\section{Involutions}

Next, we want to investigate  various anti-symplectic involutions on the moduli space $\M$.

\subsection{Involutions that satisfy Duistermaat's conditions}
$ $

In this subsection, we wish to find all involutions that satisfy the Duistermaat conditions, i.e., involutions that are compatible with the torus action and are anti-symplectic with respect to the Hamiltonian torus action. It turns out that if Duistermaat's conditions are to be satisfied, the possibilities of the involution are very limited. In fact, there is {\em only one type} that we shall now explain.

Consider a compact connected symplectic manifold $(M,\omega)$ with a Hamiltonian action of a torus $T$. Let $\tau$ be an anti-symplectic involution on $M$. Duistermaat \cite{Du83} called such an involution $\tau$ {\em compatible} with the torus action if the Hamiltonian functions $\mu_\xi$ are $\tau$-invariant: $\tau^*\mu_{\xi}=\mu_\xi$ for all $\xi\in \ft$. This is equivalent to saying that the generating vector field $\xi_M$ reverses its direction under $\tau$ for any $\xi\in\ft$, i.e. satisfies $\tau_*(\xi_M)=-\xi_M$. Duistermaat showed that the moment map image $\mu(M^\tau)$ of the fixed point set $M^\tau$ is the same as the moment map image $\mu(M)$ of the whole manifold $M$. We call such $\tau$ an involution satisfying Duistermaat's conditions. Here we will define an involution $\tau$ on $\M^\circ$ that satisfies  Duistermaat's conditions with respect to our Hamiltonian $T^3$ action and then explain that this is the only possibility.

Recall from Theorem \ref{thm:trivial} that $\mu_\Lambda:\M^\circ\to\Delta^\circ$ is a moment map and $\M^\circ$ is a trivial principal $T^3_\Lambda$-bundle over $\Delta^\circ$. Let $\fs:\Delta^\circ\to\M^\circ$ be a global section. Denote the image of the composition $\fs\circ\mu_\Lambda$ as $\fs(\mu_\Lambda([g_1;h_1;g_2;h_2]))=([g_1^\fs;h_1;g_2^\fs;h_2])$ where the superscript $\fs$ represents the choice of the section $\fs$. Notice that $h_1,h_2$ do not change since the moment map $\mu_\Lambda$ is defined using $h_1$ and $h_2$.

Now since $[g_1;h_1;g_2;h_2]$ and $[g_1^\fs;h_1;g_2^\fs;h_2]$ are in the same fiber, and $T^3_\Lambda$ acts freely on each fiber, there exists a unique $(e^{\i\lambda_1},e^{\i\lambda_2},e^{\i\lambda_3})\in T^3\cong T^3_\Lambda$, $\lambda_i\in\bR$, such that $$[e^{\lambda_3 X}g^\fs_1 e^{\lambda_1 \xi_1};h_1; e^{\lambda_3 Y}g^\fs_2 e^{\lambda_2\xi_2} ;h_2]=[g_1;h_1;g_2;h_2],$$
where $h_1=e^{\xi_1},\ h_2=e^{\xi_2},\ X=h_2h_1-(h_2h_1)^{-1},\ Y=h_1h_2-(h_1h_2)^{-1}$ . Then we define the involution to be
\[\tau([g_1;h_1;g_2;h_2])=[e^{-\lambda_3 X}g^\fs_1 e^{-\lambda_1 \xi_1};h_1; e^{-\lambda_3 Y}g^\fs_2 e^{-\lambda_2\xi_2} ;h_2] \]
One can check directly that this is an involution and that this involution is compatible with the $T^3_\Lambda$ action in the sense of Duistermaat, i.e.
 $$\tau((e^{\i\theta_1},e^{\i\theta_2},e^{\i\theta_3})\cdot [g_1;h_1;g_2;h_2])=(e^{-\i\theta_1},e^{-\i\theta_2},e^{-\i\theta_3})\cdot \tau([g_1;h_1;g_2;h_2])$$
 for all $(e^{\i\theta_1},e^{\i\theta_2},e^{\i\theta_3})\in T^3\cong T^3_\Lambda$ and $[g_1;h_1;g_2;h_2]\in\M^\circ$.

 In fact,Figure 2 shows clearly
 why such $\tau$ is an involution on $\M^\circ$
and why it is compatible with the $T^3_\Lambda$ action. This is because $\tau$ is essentially defined by this involution $$(t_1,t_2,t_3,x_1,x_2,x_3)\mapsto (t_1^{-1},t_2^{-1},t_3^{-1},x_1,x_2,x_3)$$ on $T^3_\Lambda\times\Delta^\circ$ and $\sigma(t)$ in Figure 2 is simply $\sigma(t_1,t_2,t_3)=(t_1^{-1},t_2^{-1},t_3^{-1})$. Moreover, since $\M^\circ\cong T^3_\Lambda\times \Delta^\circ$ is a completely integrable system, we see that $\tau$ is indeed anti-symplectic as \(\displaystyle \sum_{i=1}^3 dt_i\wedge dx_i\) becomes $\displaystyle -\sum_{i=1}^3 dt_i\wedge dx_i$. Thus, the involution $\tau$ we defined here satisfies all of Duistermaat's conditions.

\begin{figure}[h]\label{f:tau}
\scalebox{0.4}{\includegraphics{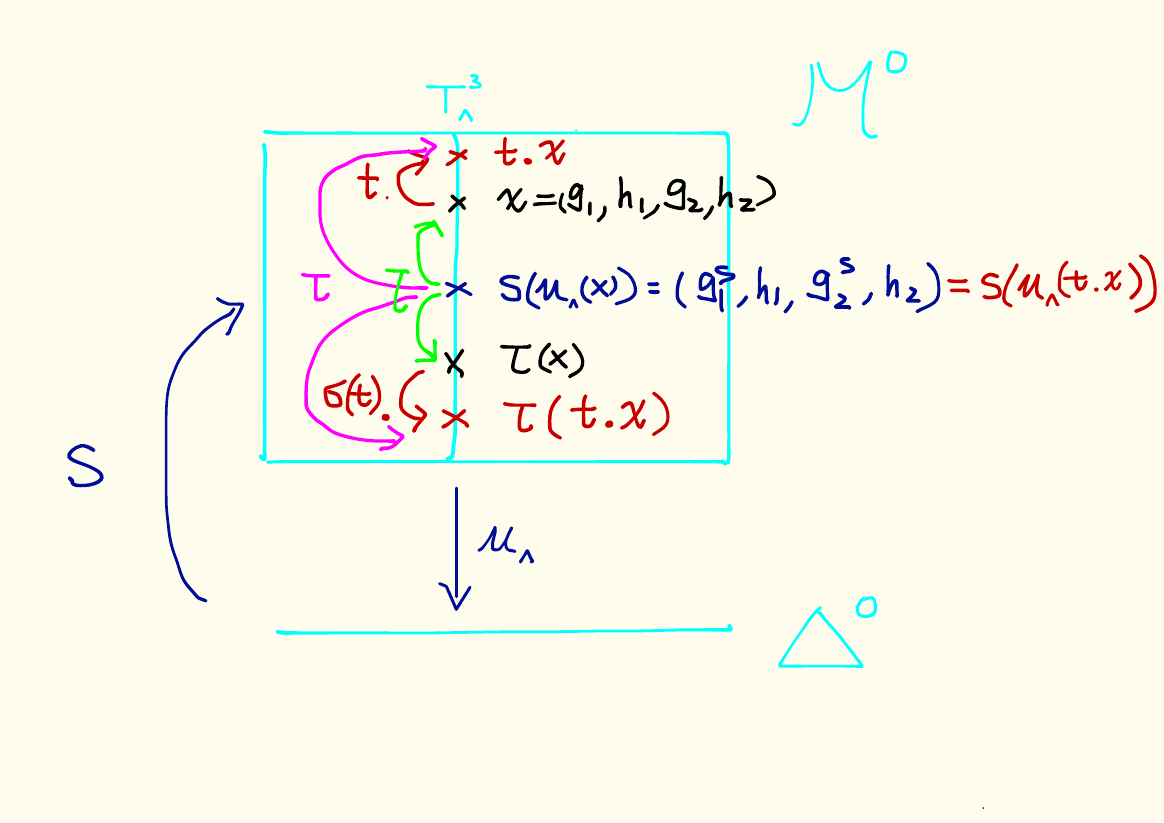}}
\caption{The involution $\tau$}
\end{figure}

From the construction, we see that this is the only involution (unique up to the choice of the section $\fs$) that can satisfy Duistermaat's condition with respect to our torus action, because in order to have the generating vector field going in the opposite direction, $(t_1,t_2,t_3)$ must go to $(t_1^{-1},t_2^{-1},t_3^{-1})$, so the only variations are the different choices of global sections (i.e. trivializations or ways to give the coordinate systems). This completes our claim.

O'Shea-Sjamaar generalized Duistermaat's result to compact Lie group $G$ as follows. Consider a compact connected symplectic manifold $(M,\omega)$ with a Hamiltonian $G$ action. Let $\tau$ be an anti-symplectic involution on $M$ and $\sigma$ be an involution on $G$. O'Shea-Sjamaar called the involution $(\tau,\sigma)$ compatible with the Hamiltonian $G$-action if $\tau(g\cdot x)=\sigma(g)\cdot \tau(x)$ for all $g\in G$ and $x\in M$ together with a moment map condition $\mu(\tau(x))=-\sigma(\mu(x))$. Under these assumptions, they showed that the moment map image of the fixed point set $M^\tau$ is the same as the moment map image of the whole manifold $M$, just as in Duistermaat's situation.
In particular, if $G$ is connected, then the moment map condition automatically follows by the condition $\tau(g\cdot x)=\sigma(g)\cdot \tau(x)$. When $G$ is an $n$-torus $T^n$, involutions $\sigma$ on $G$ are just a combination of Duistermaat's condition $t\mapsto t^{-1}$ and $(t_1,\cdots,t_i,\cdots,t_j,\cdots,t_n)\in T^n \to (t_1,\cdots,t_j,\cdots,t_i,\cdots,t_n)\in T^n$.
In our case, the $T^3$-action only acts on the first and third coordinates $g_1$, $g_2$ of points in the moduli space $\{[g_1;h_1;g_2;h_2]\}$. Thus, there exists no $\tau$ that is anti-symplectic and compatible with $\sigma$, if $\sigma:(t_1,t_2,t_3)\to(t_2,t_1,t_3)$ or $(t_3,t_2,t_1)$ or $(t_1,t_3,t_2)$. Thus, even when we relax the condition to O'Shea-Sjamaar's condition, our result remains the same, as the only compatible ones will be compatible in the sense of Duistermaat.

\subsection{Other kind of involutions, an example}
$ $

In this section, we would like to investigate another interesting involution $[g_1;h_1;g_2;h_2]\mapsto[h_2;g_2;h_1;g_1]$ on the moduli space $\M$.  It is anti-symplectic, but it is not compatible with the Hamiltonian torus action defined by the Goldman flow, in the sense of Duistermaat. We show that the fixed point set of such an involution is homeomorphic to $\P^3(\bR)$, which is the fixed point set of the natural anti-symplectic involution on $\P^3(\C)$ with respect to the Fubini-Study form. This suggests that there might be another moment  map whose Hamiltonian flows would be compatible with this involution. Moreover, since the involution is anti-symplectic, the fixed point set is Lagrangian(wherever the symplectic form is defined), where $\P^3(\bR)$ is also a Lagrangian submanifold of $\P^3(\bC)$.

Define an involution on $\Hom(\pi_1(\Si),K)$ by $\sigma:(g_1,h_1,g_2,h_2)\to (h_2,g_2,h_1,g_1)$. It induces an involution, also denoted by $\sigma$, on the moduli space $\M$, and if $[g_1;h_1;g_2;h_2]\in\M$ is fixed under $\sigma$, then there exists some $k\in K$ such that
$k\cdot(g_1,h_1,g_2,h_2)=(h_2,g_2,h_1,g_1)$.
In other words, $k^2\in K_{h_1}\cap K_{g_1}\cap K_{h_2}\cap K_{g_2}$, where $K_a$ denotes the stabilizer of $a$ in $K$.

Since $K=\SU(2)$, there are only three cases for the intersection of these stabilizers:
\begin{eqnarray*}
&& (I)~ K_{h_1}\cap K_{g_1}\cap K_{h_2}\cap K_{g_2}=Z(K) \\
&& (II)~ K_{h_1}\cap K_{g_1}\cap K_{h_2}\cap K_{g_2}=T \quad \mbox{for some maximal torus $T$} \\
&& (III)~ K_{h_1}\cap K_{g_1}\cap K_{h_2}\cap K_{g_2}=K
\end{eqnarray*}
By direct calculation, the fixed point set $\M^\sigma$ decomposes into three strata according to their stabilizers,
cases $(I), (II), (III)$, then $\M^\sigma=\M^\sigma_{(I)}\bigcup \M^\sigma_{(II)} \bigcup \M^\sigma_{(III)}$ where
\begin{eqnarray*}
\M^\sigma_{(I)}&=&\{(g,h,khk^{-1},kgk^{-1})\in K^4\mid [g,h]=k^{-1}[g,h]k,k^{2}\in Z(K),\\
&& \quad K_g\cap K_h\cap K_{kgk^{-1}}\cap K_{khk^{-1}}=Z(K)\}/K \\
\M^\sigma_{(II)}&=&\{(g,h,khk^{-1},kgk^{-1})\in K^4\mid [g,h]=k^{-1}[g,h]k, k^{2}\in T,\\
&& \quad K_g\cap K_h\cap K_{kgk^{-1}}\cap K_{khk^{-1}}=T, \mbox{for any } T\}/K \\
\M^\sigma_{(III)}&=&\{[I;I;I;I],[I;-I;-I;I],[-I;I;I;-I],[-I;-I;-I;-I]\} .
\end{eqnarray*}

Since $\SU(2)$ admits property CI (ref:\cite{S}), a representation in $\cM$ is irreducible
(in other words the stabilizer has the same dimension as the center) iff it is {\em good} (i.e. stabilizer is the center) in the sense of Johnson and Millson (ref:\cite{JM}).
So the copy $\M^\sigma_{(I)}$ contains only the irreducible representations, i.e. $\cM^\sigma_{(I)}\subset \cM^i$ and $\M^\sigma_{(II)}\bigcup\M^\sigma_{(III)}$ contains only abelian representations, i.e. $\M^\sigma_{(II)}\bigcup\M^\sigma_{(III)}\subset (\cM\setminus \cM^i)$.
However, to understand the topology of the fixed point set $\M^\sigma$, it is better to describe it differently as
$\M^\sigma=\cN_1\bigcup \cN_2$, where

\begin{eqnarray*}
\cN_1&=&\{(g,h,khk^{-1},kgk^{-1})\in K^4\mid~ [g,h]=k^{-1}[g,h]k,~ [g,h]\neq I,~ k^2\in Z(K)\}/K \\
&=& \{(g,h,h,g)\in K^4\mid~ [g,h]\neq I\}/K \\
&& \bigcup \{(g,h,khk^{-1},kgk^{-1})\in K^4\mid~ [g,h]=k^{-1}[g,h]k,~ [g,h]\neq I,~ k^2=-I\}/K
\end{eqnarray*}
and
\begin{eqnarray*}
\cN_2&=&\{(g,h,khk^{-1},kgk^{-1})\in K^4\mid~ [g,h]=I,~k^2=-I\}/K
\end{eqnarray*}

Notice that $\cN_1$ is a proper subset of $\M^\sigma_{(I)}$, while $\cN_2$ contains $\M^\sigma_{(II)}\cup \M^\sigma_{(III)}$ and some points in $\M^\sigma_{(I)}$. Though it looks like $\M^\sigma$ is disconnected with several connected components, we will show later that $\M^\sigma$ is one connected set.

The main result of this section is the following:
\begin{thm} \label{thm:RP3}
The fixed point set $\M^\sigma$ is homeomorphic to $\P^3(\bR)$.
\end{thm}

We start with the set $\cN_1$. Recall that
the quotient space $\{(g,h,h,g)\in K^4\}/K$ is a {\em fat pillow} with four vertices (ref:  Goldman \cite{Go84} \cite{Go07}), which we will refer back as the {\em Goldman pillow} (Figure 1 of Goldman \cite{Go07}) in the rest of this section. Topologically this is just a three-dimensional closed ball $B^3$. The interior of the pillow are those points with condition $[g,h]\neq I$, and the surface of the pillow are those points with condition $[g,h]=I$, while the four vertices are $[I;I;I;I]$, $[I;-I;-I;I]$, $[-I;I;I;-I]$, $[-I;-I;-I;-I]$. See  Goldman (\cite{Go84}) for more details about the pillow.
So the first piece in $\cN_1$ is the interior of this pillow.

Next we want to show that the second piece in the description of $\cN_1$ above is in one to one correspondence with the first piece except at one special point (in the interior of the pillow):
\begin{lm}\label{lm:N1_1}
There exists a one-to-one correspondence between
\[\{(g,h,h,g)\in K^4\mid~ [g,h]\neq \pm I\}/K\] and
\[\{(g,h,khk^{-1},kgk^{-1})\in K^4\mid~ [g,h]=k^{-1}[g,h]k,~ [g,h]\neq \pm I,~ k^2=-I\}/K\]
\end{lm}

\begin{proof}
If $[g,h]\neq \pm I$, then the stabilizer of $[g,h]$ is a maximal torus $T$, which intersects with the conjugacy class of $\left(
\begin{smallmatrix}
i & 0
\\
0 & -i
\end{smallmatrix}
\right)$ (i.e. the condition $k^2=-I$) at exactly two points $k$ and $-k$, i.e. there exist a unique $k$ (up to $\pm 1$) conjugate to $\left(
\begin{smallmatrix}
i & 0
\\
0 & -i
\end{smallmatrix}
\right)$ such that $[g,h]=k^{-1}[g,h]k$. In other words, $[g;h;h;g]$ and $[g;h;khk^{-1};kgk^{-1}]$ are in one to one correspondence in these two sets respectively.
\end{proof}

Now, we look at the special points with condition $[g,h]=-I$. We will show that the set $\{(g,h,h,g)\in K^4\mid [g,h]=-I\}/K$ is just one point, while the set $\{(g,h,khk^{-1},kgk^{-1})\in K^4\mid [g,h]=-I,k^{2}=-I\}/K $ is homeomorphic to $\P^2(\bR)$.

\begin{lm}\label{lm:N1_2}
The subset $\{(g,h,h,g)\in K^4\mid [g,h]=- I\}/K$ in the interior of the Goldman pillow is just one point. In fact, it is the point that maps to $(0,0,0)$ under the Goldman map $\Phi:~[(g_1,h_1,g_2,h_2)]\mapsto (\tr (h_1), \tr (h_2), \tr (h_1h_2))$.
\end{lm}

\begin{proof}
$ $

\begin{enumerate}
\item Claim $\Phi^{-1}(0,0,0)\cap \{(g,h,h,g)\in K^4\mid g,h\in K\}/K=[\left(
\begin{smallmatrix}
i & 0
\\
0 & -i
\end{smallmatrix}
\right);  \left(
\begin{smallmatrix}
0 & -1
\\
1 & 0
\end{smallmatrix}
\right);\left(
\begin{smallmatrix}
0 & -1
\\
1 & 0
\end{smallmatrix}
\right)     ;\left(
\begin{smallmatrix}
i & 0
\\
0 & -i
\end{smallmatrix}
\right)]$ is just a point:

Suppose $[g;h;h;g]\in \Phi^{-1}(0,0,0)$. Then trace$(g)=0=$ trace$(h)=$ trace$(gh)$.
Since for $\SU(2)$, trace completely determines the conjugacy classes, there exists $k\in \SU(2)$ such that $kgk^{-1}=\left(
\begin{smallmatrix}
i & 0
\\
0 & -i
\end{smallmatrix}
\right)$. So $$[g;h;h;g]=[kgk^{-1};khk^{-1};khk^{-1};kgk^{-1}]=[
\left(
\begin{smallmatrix}
i & 0
\\
0 & -i
\end{smallmatrix}
\right);khk^{-1};khk^{-1};
\left(
\begin{smallmatrix}
i & 0
\\
0 & -i
\end{smallmatrix}
\right) ],$$ and since trace$(h)=0=$ trace$(gh)$, we know that trace$(khk^{-1})=0=$ trace$(kgk^{-1}khk^{-1})$ also.

Any element of $\SU(2)$ with zero trace is of the form $\{\left(
\begin{smallmatrix}
ci & z
\\
-\bar{z} & -ci
\end{smallmatrix}
\right)\mid c^2+|z|^2=1\}$.
Denote $khk^{-1}$ by this general form $\left(
\begin{smallmatrix}
ci & z
\\
-\bar{z} & -ci
\end{smallmatrix}
\right)$ for some $c\in \bR,\ z\in\bC$ and $c^2+|z|^2=1$, then trace$(kgk^{-1}khk^{-1})=0$ implies that $0=$ trace$(\left(
\begin{smallmatrix}
i & 0
\\
0 & -i
\end{smallmatrix}
\right)\left(
\begin{smallmatrix}
ci & z
\\
-\bar{z} & -ci
\end{smallmatrix}
\right))=$ trace ($\left(
\begin{smallmatrix}
-c & iz
\\
i\bar{z} & -c
\end{smallmatrix}
\right) )=0$. Thus $khk^{-1}=\left(
\begin{smallmatrix}
0 & e^{i\theta}
\\
-e^{-i\theta} & 0
\end{smallmatrix}
\right)$.

The stabilizer of $\left(
\begin{smallmatrix}
i & 0
\\
0 & -i
\end{smallmatrix}
\right)$ is this maximal torus $\{\left(
\begin{smallmatrix}
e^{it} & 0
\\
0 & e^{-it}
\end{smallmatrix}
\right)\}$. Let $e^{it}=-e^{-i\theta/2}$, then
\[\left(
\begin{smallmatrix}
e^{it} & 0
\\
0 & e^{-it}
\end{smallmatrix}
\right)\left(
\begin{smallmatrix}
0 & e^{i\theta}
\\
-e^{-i\theta} & 0
\end{smallmatrix}
\right)\left(
\begin{smallmatrix}
e^{-it} & 0
\\
0 & e^{it}
\end{smallmatrix}
\right)=\left(
\begin{smallmatrix}
0 & e^{i\theta}e^{2it}
\\
-e^{-i\theta}e^{-2it} & 0
\end{smallmatrix}
\right)=\left(
\begin{smallmatrix}
0 & -1
\\
1 & 0
\end{smallmatrix}
\right).\]
i.e. $\left(
\begin{smallmatrix}
e^{it} & 0
\\
0 & e^{-it}
\end{smallmatrix}
\right)$ fixes $\left(
\begin{smallmatrix}
i & 0
\\
0 & -i
\end{smallmatrix}
\right)$ and conjugates $\left(
\begin{smallmatrix}
0 & e^{i\theta}
\\
-e^{-i\theta} & 0
\end{smallmatrix}
\right)$ to $\left(
\begin{smallmatrix}
0 & -1
\\
1 & 0
\end{smallmatrix}
\right)$.
We conclude that for any $(g,h,h,g)\in K^4$ satisfying $\Phi([(g,h,h,g)])=(0,0,0)$, it can always be conjugated to
$(\left(
\begin{smallmatrix}
i & 0
\\
0 & -i
\end{smallmatrix}
\right),  \left(
\begin{smallmatrix}
0 & -1
\\
1 & 0
\end{smallmatrix}
\right),\left(
\begin{smallmatrix}
0 & -1
\\
1 & 0
\end{smallmatrix}
\right)     ,\left(
\begin{smallmatrix}
i & 0
\\
0 & -i
\end{smallmatrix}
\right))$. In other words, $$\Phi^{-1}(0,0,0)\cap \{(g,h,h,g)\in K^4\}/K=[\left(
\begin{smallmatrix}
i & 0
\\
0 & -i
\end{smallmatrix}
\right);  \left(
\begin{smallmatrix}
0 & -1
\\
1 & 0
\end{smallmatrix}
\right); \left(
\begin{smallmatrix}
0 & -1
\\
1 & 0
\end{smallmatrix}
\right)     ;\left(
\begin{smallmatrix}
i & 0
\\
0 & -i
\end{smallmatrix}
\right)].$$

\item Claim $\{(g,h,h,g)\in K^4\mid [g,h]=-I\}/K=[\left(
\begin{smallmatrix}
i & 0
\\
0 & -i
\end{smallmatrix}
\right);  \left(
\begin{smallmatrix}
0 & -1
\\
1 & 0
\end{smallmatrix}
\right);\left(
\begin{smallmatrix}
0 & -1
\\
1 & 0
\end{smallmatrix}
\right) ;    \left(
\begin{smallmatrix}
i & 0
\\
0 & -i
\end{smallmatrix}
\right)]$:

If $[g,h]=-I$, then $g,h\notin Z(K)$ so there exists $k\in K$ such that $kgk^{-1}=\left(
\begin{smallmatrix}
e^{i\theta} & 0
\\
0 & e^{-i\theta}
\end{smallmatrix}
\right)$ for some $\theta\in (0,\pi)$.
So any $[g;h;h;g]$ with $[g,h]=-I$ satisfies
\[ [g;h;h;g]=[kgk^{-1};khk^{-1};khk^{-1};kgk^{-1}]=
[\left(
\begin{smallmatrix}
e^{i\theta} & 0
\\
0 & e^{-i\theta}
\end{smallmatrix}
\right);khk^{-1};khk^{-1};\left(
\begin{smallmatrix}
e^{i\theta} & 0
\\
0 & e^{-i\theta}
\end{smallmatrix}
\right)], \]  where $[\left(
\begin{smallmatrix}
e^{i\theta} & 0
\\
0 & e^{-i\theta}
\end{smallmatrix}
\right), khk^{-1}]=k[g,h]k^{-1}=[g,h]=-I$.

Let $khk^{-1}=\left(
\begin{smallmatrix}
\alpha & \beta
\\
-\bar{\beta} & \bar{\alpha}
\end{smallmatrix}
\right)\in \SU(2)$. Then $[\left(
\begin{smallmatrix}
e^{i\theta} & 0
\\
0 & e^{-i\theta}
\end{smallmatrix}
\right), \left(
\begin{smallmatrix}
\alpha & \beta
\\
-\bar{\beta} & \bar{\alpha}
\end{smallmatrix}
\right)]=-I$ gives $\left(
\begin{smallmatrix}
\alpha & \beta e^{2i\theta}
\\
-\bar{\beta}e^{-2i\theta} & \bar{\alpha}
\end{smallmatrix}
\right)=\left(
\begin{smallmatrix}
-\alpha & -\beta
\\
\bar{\beta} & \bar{\alpha}
\end{smallmatrix}
\right)$, i.e. $\alpha=0$ and $e^{2i\theta}=-1$, which implies that $khk^{-1}=\left(
\begin{smallmatrix}
0 & e^{it}
\\
-e^{-it} & 0
\end{smallmatrix}
\right)$ for some $t\in [0,\pi]$ and $kgk^{-1}=\left(
\begin{smallmatrix}
i & 0
\\
0 & -i
\end{smallmatrix}
\right)$.

Thus, as before, using only the stabilizer of $\left(
\begin{smallmatrix}
i & 0
\\
0 & -i
\end{smallmatrix}
\right)$, we may conjugate $\left(
\begin{smallmatrix}
0 & e^{it}
\\
-e^{-it} & 0
\end{smallmatrix}
\right)$ into $\left(
\begin{smallmatrix}
0 & -1
\\
1 & 0
\end{smallmatrix}
\right)$. In other words, for any $(g,h,h,g)$ satisfying $[g,h]=-I$, we can always conjugate it to $(\left(
\begin{smallmatrix}
i & 0
\\
0 & -i
\end{smallmatrix}
\right),  \left(
\begin{smallmatrix}
0 & -1
\\
1 & 0
\end{smallmatrix}
\right),\left(
\begin{smallmatrix}
0 & -1
\\
1 & 0
\end{smallmatrix}
\right)     ,\left(
\begin{smallmatrix}
i & 0
\\
0 & -i
\end{smallmatrix}
\right))$. This shows that $$\{(g,h,h,g)\in K^4\mid [g,h]=-I\}/K=[(\left(
\begin{smallmatrix}
i & 0
\\
0 & -i
\end{smallmatrix}
\right),  \left(
\begin{smallmatrix}
0 & -1
\\
1 & 0
\end{smallmatrix}
\right),\left(
\begin{smallmatrix}
0 & -1
\\
1 & 0
\end{smallmatrix}
\right)     ,\left(
\begin{smallmatrix}
i & 0
\\
0 & -i
\end{smallmatrix}
\right))].$$
\end{enumerate}

Combining (1) and (2) we complete the Lemma.
\end{proof}

\begin{lm} \label{lm:N1_3}
The set $\{(g,h,khk^{-1},kgk^{-1})\in K^4 \mid [g,h]=-I,k^2=-I\}/K$ is homeomorphic to $\P^2(\bR)$.
\end{lm}

\begin{proof}
Choose any $(g,h,khk^{-1},kgk^{-1})$ in this set. From the discussion before we know that the condition $[g,h]=-I$ implies that there exists $x\in K$ such that $xgx^{-1}=\left(
\begin{smallmatrix}
i & 0
\\
0 & -i
\end{smallmatrix}
\right), xhx^{-1}=\left(
\begin{smallmatrix}
0 & e^{it}
\\
-e^{-it} & 0
\end{smallmatrix}
\right),$  i.e. \begin{eqnarray*}
\lefteqn{
x\cdot(g,h,khk^{-1},kgk^{-1})=}\\
&& (\left(
\begin{smallmatrix}
i & 0
\\
0 & -i
\end{smallmatrix}
\right),\left(
\begin{smallmatrix}
0 & e^{it}
\\
-e^{-it} & 0
\end{smallmatrix}
\right),(xkx^{-1})\left(
\begin{smallmatrix}
0 & e^{it}
\\
-e^{-it} & 0
\end{smallmatrix}
\right)(xk^{-1}x^{-1}),(xkx^{-1})
\left(
\begin{smallmatrix}
i & 0
\\
0 & -i
\end{smallmatrix}
\right)(xk^{-1}x^{-1})).\end{eqnarray*}

Denote $k'=xkx^{-1}$. Again, there exists $y$ in the stabilizer of $\left(
\begin{smallmatrix}
i & 0
\\
0 & -i
\end{smallmatrix}
\right)$ such that
\begin{eqnarray*}
\lefteqn{y\cdot(\left(
\begin{smallmatrix}
i & 0
\\
0 & -i
\end{smallmatrix}
\right),\left(
\begin{smallmatrix}
0 & e^{it}
\\
-e^{-it} & 0
\end{smallmatrix}
\right),k'\left(
\begin{smallmatrix}
0 & e^{it}
\\
-e^{-it} & 0
\end{smallmatrix}
\right)k'^{-1},k'
\left(
\begin{smallmatrix}
i & 0
\\
0 & -i
\end{smallmatrix}
\right)k'^{-1})} \\
&&= (\left(
\begin{smallmatrix}
i & 0
\\
0 & -i
\end{smallmatrix}
\right),\left(
\begin{smallmatrix}
0 & -1
\\
1 & 0
\end{smallmatrix}
\right),(yk'y^{-1})\left(
\begin{smallmatrix}
0 & -1
\\
1 & 0
\end{smallmatrix}
\right)(yk'^{-1}y^{-1}),(yk'y^{-1})
\left(
\begin{smallmatrix}
i & 0
\\
0 & -i
\end{smallmatrix}
\right)(yk'^{-1}y^{-1})).
\end{eqnarray*}

In other words, any element in this set can be conjugated into the following form:
\[(\left(
\begin{smallmatrix}
i & 0
\\
0 & -i
\end{smallmatrix}
\right),\left(
\begin{smallmatrix}
0 & -1
\\
1 & 0
\end{smallmatrix}
\right),k''\left(
\begin{smallmatrix}
0 & -1
\\
1 & 0
\end{smallmatrix}
\right)(k'')^{-1},k''
\left(
\begin{smallmatrix}
i & 0
\\
0 & -i
\end{smallmatrix}
\right)(k'')^{-1})\]
for some $k''=yk'y^{-1}=yxkx^{-1}y^{-1}$, and since the condition that $k^2=-I$ is equivalent to the condition that $k$ is conjugate to $\left(
\begin{smallmatrix}
i & 0
\\
0 & -i
\end{smallmatrix}
\right)$, so $k''$ is conjugate to $\left(
\begin{smallmatrix}
i & 0
\\
0 & -i
\end{smallmatrix}
\right)$ also.

In fact, any two elements of the above form are conjugate iff their $k''$  differ by an element in the intersection of their stabilizers $K_{\left(
\begin{smallmatrix}
0 & -1
\\
1 & 0
\end{smallmatrix}
\right)}\cap K_{\left(
\begin{smallmatrix}
i & 0
\\
0 & -i
\end{smallmatrix}
\right)}=Z(K)=\{\pm I\}$. To show this, suppose that
\begin{eqnarray*}
\lefteqn{
z\cdot(\left(
\begin{smallmatrix}
i & 0
\\
0 & -i
\end{smallmatrix}
\right),\left(
\begin{smallmatrix}
0 & -1
\\
1 & 0
\end{smallmatrix}
\right),k''_1\left(
\begin{smallmatrix}
0 & -1
\\
1 & 0
\end{smallmatrix}
\right)(k''_1)^{-1},k''_1
\left(
\begin{smallmatrix}
i & 0
\\
0 & -i
\end{smallmatrix}
\right)(k''_1)^{-1}) }\\
&=&
 (\left(
\begin{smallmatrix}
i & 0
\\
0 & -i
\end{smallmatrix}
\right),\left(
\begin{smallmatrix}
0 & -1
\\
1 & 0
\end{smallmatrix}
\right),k''_2\left(
\begin{smallmatrix}
0 & -1
\\
1 & 0
\end{smallmatrix}
\right)(k''_2)^{-1},k''_2
\left(
\begin{smallmatrix}
i & 0
\\
0 & -i
\end{smallmatrix}
\right)(k''_2)^{-1})
\end{eqnarray*}
for some $z\in K$, then $z\in K_{\left(
\begin{smallmatrix}
0 & -1
\\
1 & 0
\end{smallmatrix}
\right)}\cap K_{\left(
\begin{smallmatrix}
i & 0
\\
0 & -i
\end{smallmatrix}
\right)}$ and $(k''_2)^{-1}z k''_1\in K_{\left(
\begin{smallmatrix}
0 & -1
\\
1 & 0
\end{smallmatrix}
\right)}\cap K_{\left(
\begin{smallmatrix}
i & 0
\\
0 & -i
\end{smallmatrix}
\right)}$. Thus $(k''_2)^{-1}k''_1\in K_{\left(
\begin{smallmatrix}
0 & -1
\\
1 & 0
\end{smallmatrix}
\right)}\cap K_{\left(
\begin{smallmatrix}
i & 0
\\
0 & -i
\end{smallmatrix}
\right)}=Z(K)=\{\pm I\}$, i.e. $k''_1=\pm k''_2$.
Thus, any two such elements of the above form are conjugate iff $k''_1=\pm k''_2$.
On the other hand, the conjugacy class of $\left(
\begin{smallmatrix}
i & 0
\\
0 & -i
\end{smallmatrix}
\right)$ is homeomorphic to $S^2$, we conclude $$\{(g,h,khk^{-},kgk^{-1})\in K^4\mid [g,h]=-I,k^2=-I\}/K\cong S^2/\pm \cong \P^2(\bR).$$

\end{proof}

Combining Lemma \ref{lm:N1_1}, Lemma \ref{lm:N1_2}, and Lemma \ref{lm:N1_3} we see what $\cN_1$ looks like. It has two pieces: one is the interior of the Goldman pillow $\{(g,h,h,g)\in K^4\}/K$, and the other is a piece that is in one to one correspondence with the first piece minus an interior point and at that missing point one attaches back an $\P^2(\bR)$.

Next, we will look at the set $\cN_2$. We will show that $\cN_2$ contains the surface of the Goldman pillow in $\cN_1$, the surface of the {\em one-point-blow-up} pillow in $\cN_1$, and
lines that connect  every point on these two surfaces, i.e. $\cN_2$ connects the two pieces in $\cN_1$. Thus, $\M^\sigma=\cN_1\cup\cN_2$ is path connected.

\begin{lm} \label{lm:N2}
For any fixed $(g,h)$, the subset $\{(g,h,khk^{-1},kgk^{-1})\in K^4\mid [g,h]=I,k^2=-I\}/K\cong S^2/S^1$ is a closed interval except for $(g,h)=(\pm I,\pm I)$ where the 4 intervals degenerate to 4 points. To be precise, for each fixed $(g,h)\neq (\pm I,\pm I)$, there exists an interval
$$ \fI_{(g,h)}:[0,1]\to \{(g,h,khk^{-1},kgk^{-1})\in K^4\mid [g,h]=I,k^2=-I\}/K$$
such that $\fI_{(g,h)}(0)$ belongs to the surface of the Goldman pillow and $\fI_{(g,h)}(1)$ belongs to the surface of the {\em one-point-blow-up} pillow, and $\fI_{(g,h)}(t)$ has no intersection with either surface for any $t\in (0,1)$. For $(g,h)=(\pm I, \pm I)$, the corresponding subset $\{(g,h,khk^{-1},kgk^{-1})\in K^4\mid [g,h]=I,k^2=-I\}/K$ is a point, in fact, they are the four vertices of both pillows.
\end{lm}

\begin{proof}
\begin{enumerate}
\item Case 1: Assume that $(g,h)\neq (\pm I,\pm I)$.

Let $(g,h,khk^{-1},kgk^{-1})$ be such that $[g,h]=I$, $k^2=-I$, so $g$, $h$, belong to the same maximal torus, say $T'$, then there exists $\epsilon\in K$ such that it conjugates $T'$ to the diagonal maximal torus $T$, i.e.
\begin{eqnarray*}
\lefteqn{\epsilon\cdot(g,h,khk^{-1},kgk^{-1})}\\
&&=( \left(
\begin{smallmatrix}
e^{i\theta} & 0
\\
0 & e^{-i\theta}
\end{smallmatrix}
\right),\left(
\begin{smallmatrix}
e^{is} & 0
\\
0 & e^{-is}
\end{smallmatrix}
\right),(\epsilon k\epsilon^{-1})\left(
\begin{smallmatrix}
e^{is} & 0
\\
0 & e^{-is}
\end{smallmatrix}
\right)(\epsilon k^{-1}\epsilon^{-1}),(\epsilon k\epsilon^{-1})
\left(
\begin{smallmatrix}
e^{i\theta} & 0
\\
0 & e^{-i\theta}
\end{smallmatrix}
\right)(\epsilon k^{-1}\epsilon^{-1}))
\end{eqnarray*}
where $(\epsilon k\epsilon^{-1})^2=\epsilon k^2 \epsilon^{-1}=-I$.
In order for $g$ and $h$ to stay in this $T$, one must only use conjugation in $N(T)$.
This shows that \begin{eqnarray*}
\lefteqn{\{(g,h,khk^{-1},kgk^{-1}\in K^4\mid [g,h]=I, k^2=-I\}/K=}\\
&& \{
(\left(
\begin{smallmatrix}
e^{i\theta} & 0
\\
0 & e^{-i\theta}
\end{smallmatrix}
\right),\left(
\begin{smallmatrix}
e^{is} & 0
\\
0 & e^{-is}
\end{smallmatrix}
\right),k\left(
\begin{smallmatrix}
e^{is} & 0
\\
0 & e^{-is}
\end{smallmatrix}
\right)k^{-1},k
\left(
\begin{smallmatrix}
e^{i\theta} & 0
\\
0 & e^{-i\theta}
\end{smallmatrix}
\right)k^{-1})
\mid k^2=-I\}/N(T).\end{eqnarray*}
Now suppose we have two elements of this form that are conjugate to each other, i.e.
\begin{eqnarray*}\lefteqn{(\left(
\begin{smallmatrix}
e^{i\theta} & 0
\\
0 & e^{-i\theta}
\end{smallmatrix}
\right),\left(
\begin{smallmatrix}
e^{is} & 0
\\
0 & e^{-is}
\end{smallmatrix}
\right),k_2\left(
\begin{smallmatrix}
e^{is} & 0
\\
0 & e^{-is}
\end{smallmatrix}
\right)k_2^{-1},k_2
\left(
\begin{smallmatrix}
e^{i\theta} & 0
\\
0 & e^{-i\theta}
\end{smallmatrix}
\right)k_2^{-1})}\\
&& =t\cdot
(\left(
\begin{smallmatrix}
e^{i\theta'} & 0
\\
0 & e^{-i\theta'}
\end{smallmatrix}
\right),\left(
\begin{smallmatrix}
e^{is'} & 0
\\
0 & e^{-is'}
\end{smallmatrix}
\right),k_1\left(
\begin{smallmatrix}
e^{is'} & 0
\\
0 & e^{-is'}
\end{smallmatrix}
\right)k_1^{-1},k_1
\left(
\begin{smallmatrix}
e^{i\theta'} & 0
\\
0 & e^{-i\theta'}
\end{smallmatrix}
\right)k_1^{-1})\end{eqnarray*}
for some $t\in N(T)$,
then this implies that $$k_2^{-1}tk_1t^{-1}\in K_{\left(
\begin{smallmatrix}
e^{i\theta} & 0
\\
0 & e^{-i\theta}
\end{smallmatrix}
\right)}\cap K_{\left(
\begin{smallmatrix}
e^{is} & 0
\\
0 & e^{-is}
\end{smallmatrix}
\right)}=\{\left(
\begin{smallmatrix}
e^{i\alpha} & 0
\\
0 & e^{-i\alpha}
\end{smallmatrix}
\right)\},$$
since at least one of $\left(
\begin{smallmatrix}
e^{i\theta} & 0
\\
0 & e^{-i\theta}
\end{smallmatrix}
\right)$, $\left(
\begin{smallmatrix}
e^{is} & 0
\\
0 & e^{-is}
\end{smallmatrix}
\right)$
is not $\pm I$.
Thus $k_2^{-1}tk_1t^{-1}=\left(
\begin{smallmatrix}
e^{i\alpha} & 0
\\
0 & e^{-i\alpha}
\end{smallmatrix}
\right)$ for some $\alpha$, and
$t^{-1}k_2^{-1}tk_1=t^{-1}\left(
\begin{smallmatrix}
e^{i\alpha} & 0
\\
0 & e^{-i\alpha}
\end{smallmatrix}
\right)t=\left(
\begin{smallmatrix}
e^{i\alpha} & 0
\\
0 & e^{-i\alpha}
\end{smallmatrix}
\right) \mbox{or} \left(
\begin{smallmatrix}
e^{-i\alpha} & 0
\\
0 & e^{i\alpha}
\end{smallmatrix}
\right)$ because $t\in N(T)$.

Since $k_1$, $k_2$ satisfy $k^2=-I$, we may assume
$k_1= \left(
\begin{smallmatrix}
0 & 1
\\
-1 & 0
\end{smallmatrix}
\right)$, $k_2=g\left(
\begin{smallmatrix}
0 & 1
\\
-1 & 0
\end{smallmatrix}
\right)g^{-1}$, and try to find $k_2$ that satisfies the condition
\begin{displaymath}
t^{-1}k_2^{-1}t= \left\{ \begin{array}{llll}
\left(
\begin{smallmatrix}
e^{i\alpha} & 0
\\
0 & e^{-i\alpha}
\end{smallmatrix}
\right)k_1^{-1} &=&\left(
\begin{smallmatrix}
0&-e^{i\alpha}
\\
e^{-i\alpha}&0
\end{smallmatrix}
\right) &\textrm{if $t\in T\cap N(T)$,}\\
\left(
\begin{smallmatrix}
e^{-i\alpha} & 0
\\
0 & e^{i\alpha}
\end{smallmatrix}
\right)k_1^{-1} &=&\left(
\begin{smallmatrix}
0&-e^{-i\alpha}
\\
e^{i\alpha}&0
\end{smallmatrix}
\right) &\textrm{otherwise.}
\end{array} \right.
\end{displaymath}

Thus, \begin{displaymath}
g\left(
\begin{smallmatrix}
0 & -1
\\
1 & 0
\end{smallmatrix}
\right) g^{-1}=k_2^{-1}= \left\{ \begin{array}{ll}
t\left(
\begin{smallmatrix}
0&-e^{i\alpha}
\\
e^{-i\alpha}&0
\end{smallmatrix}
\right)t^{-1} &\textrm{if $t\in T\cap N(T)$}\\
t\left(
\begin{smallmatrix}
0&-e^{-i\alpha}
\\
e^{i\alpha}&0
\end{smallmatrix}
\right)t^{-1} &\textrm{otherwise}
\end{array} \right.
\end{displaymath}

So the equation for $g$ can  indeed be solved: for example,
$g=
\left(
\begin{smallmatrix}
0&e^{it+\frac{i\alpha}{2}}
\\
-e^{-it-\frac{i\alpha}{2}}&0
\end{smallmatrix}
\right)$ if
$t=\left(
\begin{smallmatrix}
e^{it}&0
\\
0&e^{-it}
\end{smallmatrix}
\right)$, and $g=
\left(
\begin{smallmatrix}
0&e^{-it-\frac{i\alpha}{2}}
\\
-e^{it-\frac{i\alpha}{2}}&0
\end{smallmatrix}
\right)$ if
$t=\left(
\begin{smallmatrix}
0 & 1
\\
-1 & 0
\end{smallmatrix}
\right)\left(
\begin{smallmatrix}
e^{it}&0
\\
0&e^{-it}
\end{smallmatrix}
\right)$ or $\left(
\begin{smallmatrix}
0 & -1
\\
1 & 0
\end{smallmatrix}
\right)\left(
\begin{smallmatrix}
e^{it}&0
\\
0&e^{-it}
\end{smallmatrix}
\right)$.

Thus, if $k_1=\left(
\begin{smallmatrix}
0 & 1
\\
-1 & 0
\end{smallmatrix}
\right)$, then with any $k_2\in\{\left(
\begin{smallmatrix}
0 & e^{i\beta}
\\
-e^{-i\beta} & 0
\end{smallmatrix}
\right)\}$ we will have two conjugate representations defined using $k_1$ and $k_2$ respectively.
In other words, for any fixed pair $(g,h)=
(\left(
\begin{smallmatrix}
 e^{i\theta} &0
\\
0&e^{-i\theta}
\end{smallmatrix}
\right), \left(
\begin{smallmatrix}
 e^{is} &0
\\
0&e^{-is}
\end{smallmatrix}
\right))\neq(\pm I, \pm I)$, the set
\begin{eqnarray*}
&\{(\left(
\begin{smallmatrix}
 e^{i\theta} &0
\\
0&e^{-i\theta}
\end{smallmatrix}
\right), \left(
\begin{smallmatrix}
 e^{is} &0
\\
0&e^{-is}
\end{smallmatrix}
\right),k\left(
\begin{smallmatrix}
 e^{is} &0
\\
0&e^{-is}
\end{smallmatrix}
\right)k^{-1},k\left(
\begin{smallmatrix}
 e^{i\theta} &0
\\
0&e^{-i\theta}
\end{smallmatrix}
\right)k^{-1})\mid k^2=-I\}/N(T)& \\
& \cong \{k^2=-I\}/\{\left(
\begin{smallmatrix}
0 & e^{i\beta}
\\
-e^{-i\beta} & 0
\end{smallmatrix}
\right)\}\cong S^2/ S^1&\end{eqnarray*} is a closed interval,
and its two endpoints are $[\left(
\begin{smallmatrix}
 e^{i\theta} &0
\\
0&e^{-i\theta}
\end{smallmatrix}
\right); \left(
\begin{smallmatrix}
 e^{is} &0
\\
0&e^{-is}
\end{smallmatrix}
\right);\left(
\begin{smallmatrix}
 e^{is} &0
\\
0&e^{-is}
\end{smallmatrix}
\right);\left(
\begin{smallmatrix}
 e^{i\theta} &0
\\
0&e^{-i\theta}
\end{smallmatrix}
\right)]$ and $[\left(
\begin{smallmatrix}
 e^{i\theta} &0
\\
0&e^{-i\theta}
\end{smallmatrix}
\right); \left(
\begin{smallmatrix}
 e^{is} &0
\\
0&e^{-is}
\end{smallmatrix}
\right);\left(
\begin{smallmatrix}
 e^{-is} &0
\\
0&e^{is}
\end{smallmatrix}
\right);\left(
\begin{smallmatrix}
 e^{-i\theta} &0
\\
0&e^{i\theta}
\end{smallmatrix}
\right)]$, which belong to the surfaces of the two pieces of $\cN_1$ respectively.

\item Case 2: Assume that $(g,h)=(\pm I,\pm I)$.

For any fixed $(g,h)=(\pm I, \pm I)$, the corresponding subset $\{(g,h,khk^{-1},kgk^{-1})\in K^4\mid [g,h]=I, k^2=-1\}/K$ degenerates to just one of the four points $[I;I;I;I]$, $[I;-I;-I;I]$, $[-I;I;I;-I]$, $[-I;-I;-I;-I]$.
\end{enumerate}
\end{proof}

Now Theorem \ref{thm:RP3} follows from the above four lemmas which describes the set $\cN_1$ and $\cN_2$ completely:
\begin{proof}[Proof of Theorem \ref{thm:RP3}]
The above four lemmas show that (1) $\cN_1$ contains two components. One is topologically the interior of an three dimensional ball $B^3$ (the Goldman pillow). The other piece is the interior of another three dimensional ball with one interior point blown up (remove one interior point and attach back a $\P^2(\bR)$). (2) $\cN_2$ contains the surfaces of the two balls of $\cN_1$ and lines that connect all the points on one surface to the other. Thus, $\M^\sigma=\cN_1\cup\cN_2$ is topologically $\P^3(\bR)$.
\end{proof}
\begin{rem}
Points in $\cN_1$ are all irreducible representations, i.e.  $\cN_1 \subset \M^i$, while $\cN_2\cap \cM^i\neq \emptyset$ and $\cN_2\cap(\cM\setminus\cM^i)\neq\emptyset$. To be precise,
abelian representations in $\cN_2$, in other words $\cN_2\cap(\cM\setminus\cM^i)$, are the two surfaces of the two topological balls $B^3$ (the pillows) described above, and irreducible representations in $\cN_2$ are all the open intervals $\fI_{(g,h)}((0,1))$, $(g,h)\neq (\pm I,\pm I)$ defined in Lemma \ref{lm:N2}.
\end{rem}


\begin{thebibliography}{99}

\bibitem{Ch11} Choi, Suhyoung Spherical triangles and the two components of the SO(3)-character space of the fundamental group of a closed surface of genus 2. {\em Internat. J. Math.} {\bf 22} (2011), no. 9, 1261-1364.

\bibitem{Du83} Duistermaat, J. J. Convexity and tightness for restrictions of Hamiltonian functions to fixed point sets of an antisymplectic involution. {\em Trans. Amer. Math. Soc.} {\bf 275} (1983), no. 1, 417-429.

\bibitem{Go84} Goldman, William M. The symplectic nature of fundamental groups of surfaces. {\em Adv. in Math.} 54 (1984), no. 2, 200-225.

\bibitem{Go86} Goldman, William M. Invariant functions in Lie groups and Hamiltonian flows of surface group representations. {\em Invent. Math.} {\bf 85} (1986), 263-302



\bibitem{Go97} Goldman, William M. Ergodic theory on moduli spaces. {\em Ann. of Math. (2)} 146 (1997), no. 3, 475-507.

\bibitem{Go07} Goldman, William M. An ergodic action of the outer automorphism group of a free group. {\em Geom. Funct. Anal.} 17 (2007), no. 3, 793-805.

\bibitem{Go09} Goldman, William M. Trace coordinates on Fricke spaces of some simple hyperbolic surfaces. {\em Handbook of Teichm\"uller theory. Vol. II}, 611-684, IRMA Lect. Math. Theor. Phys., {\bf 13}, Eur. Math. Soc., Z\"urich, 2009.

\bibitem{JW92} Jeffrey, Lisa C.; Weitsman, Jonathan Bohr-Sommerfeld orbits in the moduli space of flat connections and the Verlinde dimension formula. {\em Comm. Math. Phys.} 150 (1992), no. 3, 593-630.

\bibitem{JW94} Jeffrey, Lisa C.; Weitsman, Toric Structures on the Moduli space of Flat Connections on a Riemann Surface: Volumes and the moment map. {\em Adv. in Math.} 106 (1994), 151-168.

\bibitem{JM} D. Johnson and J.J. Millson, Deformation spaces associated to compact hyperbolic manifolds, in: Discrete groups in geometry and
analysis (New Haven, CT, 1984), ed. R. Howe, Progr. Math., 67, pp. 48-106, Birkh\"auser Boston, Boston, MA (1987)

\bibitem{NR69} Narasimhan, M. S.; Ramanan, S. Moduli of vector bundles on a compact Riemann surface. {\em Ann. of Math. (2)} {\bf 89} 1969 14-51.

\bibitem{NS65} Narasimhan, M. S.; Seshadri, C. S. Stable and unitary vector bundles on a compact Riemann surface. {\em Ann. of Math. (2)} {\bf 82} 1965 540-567.

\bibitem{S} A.S. Sikora, Character Varieties. {\em Trans. Amer. Math. Soc.}  {\bf 364} (2012) 5173-5208.








\end{thebibliography}
\end{document}